\documentclass[letterpaper, 10 pt, conference]{ieeeconf}  

\IEEEoverridecommandlockouts                              

\usepackage{xcolor,graphicx,graphics,epsfig,float}
\usepackage{amsfonts,amsmath,amssymb,bm,bbm}
\usepackage{hyperref}
\usepackage{fancyhdr}
\usepackage{subfigure}
\usepackage{tikz}

\usepackage{lineno}  
\def\NZQ{\Bbb}               

\def\RR{{\NZQ R}}

\newtheorem{Theorem}{Theorem}
\newtheorem{Assumption}{Assumption}

\newtheorem{Definition}{Definition}
\newtheorem{Lemma}{Lemma}

\newtheorem{Remark}{Remark}
\newcommand{\mc}{\mathcal}

\title{\LARGE \bf
A mean field approach to model flows of agents with path preferences over a network 
}

\author{Fabio Bagagiolo$^{1}$, Rosario Maggistro$^{2}$, Raffaele Pesenti$^{2}$ 
	\thanks{*This work was partially supported by fund POR FSE 2014-2020 - CUP H76C18000250005 (codice 2120-19-11-2018).}
	\thanks{$^{1}$Department of Mathematics, Universit\`{a} di Trento, Via Sommarive, 14, 38123 Povo, Trento, Italy  {\tt\small fabio.bagagiolo@unitn.it}.}
	\thanks{$^{2}$Department of Management, Universit\`{a} Ca' Foscari Venezia,
		Fondamenta S. Giobbe, 873, 30121 Cannaregio, Venezia, Italy
		{\tt\small rosario.maggistro@unive.it, pesenti@unive.it}.}
}

\begin{document}

\maketitle
\thispagestyle{empty}
\pagestyle{empty}

\begin{abstract}
In this paper, we address the problem
of modeling the traffic flow of a heritage city
whose streets are represented by a network. We consider a mean field approach where the standard forward backward system of equations is  also intertwined with a path preferences dynamics. The path preferences are influenced by the congestion status on the whole
network as well as the possible hassle of being forced to run during the tour. 
We prove the existence of a mean field equilibrium as a fixed point, over a suitable set of time-varying distributions, of a map obtained as a limit of a sequence of approximating functions. Then, a bi-level optimization problem is formulated for an external controller who aims to induce a specific mean field  equilibrium.\\
\end{abstract}


\section{INTRODUCTION}

	In the recent years, the continuous growth of traffic flow  and the resulting overcrowding have led some cities 
to seek solutions to manage this phenomenon.
The crowd motion modeling and the study of flow dynamics have become in the last decades two of the main targets of the transportation research community. Different modeling approaches have been proposed which can generally be
classified into two categories: microscopic models and macroscopic
models. The former include the cellular automaton
model, the social force model, and the lattice
gas
model (see e.g., \cite{Burstedde}--\cite{Guo}),
and are particularly well suited for use
with small crowds. Macroscopic models, in contrast, focus
on the overall behavior of pedestrian flows and are more
applicable to investigations of extremely large crowds, especially
when examining aspects of motion in which individual
differences are less important \cite{Hughes}--\cite{Hoogendoorn20041}.
In this paper, starting from the results in \cite{bagpes}--\cite{bafama}, we introduce a Mean Field (MF) approach to modeling and analytically studying
the flow of daily agents
along the street of a heritage city.
Mean field games (MFG) theory goes back to the seminal work by Lasry-Lions \cite{LLions} (see also \cite{HCAINMAL}). This theory
includes methods and techniques to study differential games with a large
population of rational players and it is based on the assumption that the
population influences individuals' strategies through mean field
parameters. Several application domains such as economic, physics, biology
and network engineering accommodate MFG theoretical models (see
\cite{AcCamDolc}--\cite{Guent2011}). In particular, models to 
study of dynamics on networks and/or pedestrian movement can be found for
example in \cite{CCMar}--\cite{BBMZop}.
In this paper,  
beside the usual framing of mean field games (which is typically defined by the pair made of Hamilton-Jacobi-Bellman and transport
equations), we consider the agent's path preferences dynamics. In particular, we assume that it evolves following a perturbed best response to global information about the congestion status of the whole network and to the control vector. Moreover,
this path preferences
dynamics evolves at a slow time scale as compared
to the physical dynamics. 
We apply these arguments to the possible paths considering a network topology as in Figure~\ref{graphtopology}.

Our main result shows the existence of a MF equilibrium for our framework. This equilibrium is a time-varying distribution of agents,  $\widetilde{\rho}(t)$ for $t\in[0,T]$, on the network.
Distribution  $\widetilde{\rho}(t)$, when plugged in the cost to minimize, generates an optimal control $%
\widetilde{u}(t)$, for any agent starting at the origin $o$ 
which, in turn, yields the path preference $z(t)$ providing the time-varying distribution~$\widetilde{\rho}$.
We also introduce a possible bi-level optimization
problem for an external controller who aims to force the equilibrium to be as close as possible (in uniform topology) to a reference value $\widetilde{\rho}$.
We suppose that the external controller (the city hall, for example) may act on the congestion functions choosing them among a suitable set of admissible functions.

The rest of this paper is organized as follows. In Section \ref{sec:2}, we describe the model and state the hypotheses used in the paper. In Section \ref{sec:3}, we claim the existence of a MF equilibrium
 (due to the space limitation, we only include the sketches of the proofs) 
 and, in Section \ref{sec:4}, we define a bi-level optimization problem.
In Section \ref{sec:5}, we draw conclusions and suggests future works.

\section{MODEL DESCRIPTION}\label{sec:2}

\subsection{Network characteristics}

We model the topology of the network as a directed multi-graph $\mathcal{G=(V, E)}$, where $\mathcal{V}$ is a finite set of nodes and $\mathcal{E}$ is a finite set of directed links. Each link $e = (\nu_e,\kappa_e)$ in $\mc E$ is directed from its tail node $\nu_e$ to its head node $\kappa_e\neq \nu_e$. 
An oriented path from a node $v_0$ to a node $v_{r}$ is an ordered set of $r$ adjacent links $p=(e_1, e_2, \ldots, e_{r})$ such that $\nu_{e_1}=v_0$, $\kappa_{e_r}=v_{r}$, $v_s=\kappa_{e_{s}}=\nu_{e_{s+1}}$ for $1\le s\le r-1$, and no node is visited twice, i.e., $v_{l}\ne v_s$ for all $0\le l<s\le r$, except possibly for $v_0=v_r$, in which case the path is referred to as a cycle. Moreover, let $\ell_{e}$ be the length of the link $e \in \mc E $  and let $\ell^p=\sum_{e \in p} \ell_{e}$ be the length of the path $p$. 
%
A node $v_j$ is said to be reachable from another node $v_k$ if there exists at least a path from $v_k$ to $v_j$.

We hold the following assumptions on the multi-graph~$\mathcal{G}$ and on the agents that move along its links.
\begin{Assumption}\label{ass:main}~
	\begin{enumerate}
		\item\label{ass:main1} $\mathcal{G}$ contains no cycles.
		\item\label{ass:main2} $\mathcal{G}$ includes an origin node $o$, from which any node in $\mathcal{V}$ can be reached, and a destination node $d\ne o$, which is reachable from any node in $\mathcal{V}$, $o$~included.
		\item\label{ass:main3} Agents arrive at~$o$ in the morning  and desire to leave from~$d$ in the afternoon.
	\end{enumerate}
\end{Assumption}

We denote $\Gamma$ as the set of all the paths from $o$ to $d$. 
In particular, in this paper, we consider a  graph~$\mathcal{G}$ on which agents have only three possible paths to reach the destination~$d$ starting from the origin $o$ (see Figure~\ref{graphtopology}).
\begin{figure}[htpb]
	\centering
	\includegraphics [width=0.35\textwidth]{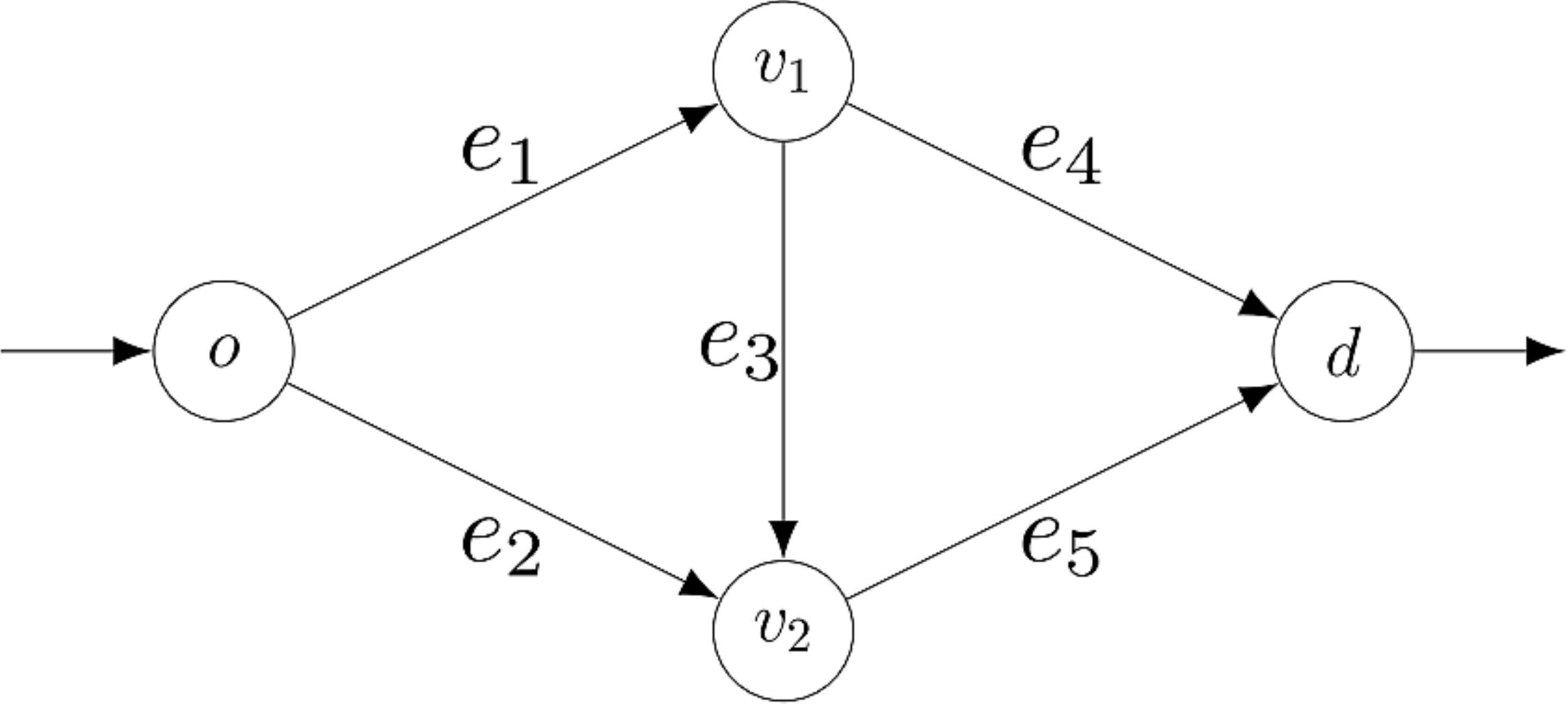}
	\caption{\label{graphtopology} The graph topology used in the paper.}
\end{figure}
In this case, the set of path is $\Gamma=\{p_1, p_2, p_3\}$, where $p_1=(e_1,e_4)$, $p_2=~(e_2, e_5)$, $p_3=(e_1, e_3, e_5)$. We denote the corresponding link-path incidence matrix by $A\in\{0,1\}^{|\mc E| \times |\Gamma|}$ with entries 
\begin{equation*}
A_{ep}=
\begin{cases}
1 & \text{if} \quad e \in p, \\
0  & \text{otherwise}.
\end{cases} 
\end{equation*}
For every link $e \in \mc E$ and time
instant $t \in [0, T]$, we denote the current mass and flow by $\rho_e(t)$ and $f_e(t)$, respectively, defined as
\begin{equation*}
\rho_e:[0, T]\to [0, \rho_{\max}], \quad f_e:[0, T]\to [0, C_e],
\end{equation*}
with $T>0$ the final horizon, i.e., the time by which every agent has to reach the destination $d$, and $C_e$ is the maximum flow capacity. Moreover let
\begin{equation}
\rho(t) := \{\rho_e(t): e \in \mc E\},\quad f(t) := \{f_e(t) : e \in \mc E \} 
\end{equation} 
be the vectors of masses and flows, respectively.
\subsection{Agents' dynamics and costs}

\begin{equation}  \label{eq:thetau}
\begin{cases}
\dot{\theta}_e(s)=u^e(s), &s\in ]t,T],\, \forall e \in \mc E,\\
\theta_e(t)=0, & \forall t \in [0, T],\, \forall e \in \mc E,
\end{cases}
\end{equation}
where $\theta_e(s) \in [0, \ell_e]$, being $\ell_e$ the length of the link. Using this space-coordinates, $\theta_e(s)=0$ means that the agent is in $\nu_e$, while $\theta_e(s)=\ell_e$ means that the agent is in $\kappa_e$ and hence he is inside the link $e$ as long as $0\leq\theta(s)\leq \ell_e$. Note that \eqref{eq:thetau} describes the evolution of a possible agent which is in $\nu_e$ at time $t$ independently of the fact that he is present at $\nu_e$ at that time. By definition, $\theta_e(s)$ and $u^e(s)$ are equal to $NaN$ when the agent is not on link $e$. Hereinafter, $NaN$ stands for {\it not a number}.

The control, $s\mapsto u^e(s)\in \mathbb{R}_{+}$, is  measurable and locally integrable, namely $u^e\in L^1_{loc}(t,T) \ \forall t$. For ease of notation, from now on, we call $\mc U_e$ the set of these kind of controls.
There is no loss of optimality in assuming $u^e\geq 0$ as we discuss later in Remark 1.
The cost to be minimized by every agent crossing a link $e$,
takes into account: i) the possible hassle of running in the link $e$ to reach $d$ on time; ii) the pain of being entrapped in highly congested link; 
iii) the disappointment of not being able to reach $d$ by the final horizon~$T$.
Such a cost can be analytically represented by 
\begin{align}\label{eqcost}
J_e(t, u^e)=&\int_t^T\chi_{\{0\leq\theta_e(s)\leq \ell_{e}\}}\left(\frac{(u^e(s))^2}{2}+\varphi_e(\rho_e(s))\right)\,ds\nonumber \\  + 
&\chi_{\{0\leq \theta_e(T)< \ell_e\}}\alpha\sum_{j\in p(e)} \ell_{j},
\end{align}
where $p(e)$ is the shortest path from $\nu_e$ to $d$, $\alpha>0$ is a constant parameter representing a cost per unit of length and $\chi$ is the characteristic function.
In~(\ref{eqcost}),
the quadratic term inside the integral stands for cost i) while the other term stands for the congestion cost ii) given by the function $\varphi_e:[0, \rho_{\max}]\to [0, +\infty[ $; the last addendum stands for cost iii).\\
We define, with a little abuse of terminology, as value functions the following:
\begin{equation}\label{Valuefunc1}
\begin{split}
V^e(t)= &\inf\limits_{\substack{u^e \in \mc U_e}}
\Big\{J_e(t, u^e) \\
& +\min\limits_{\substack {j\in \mc E : \nu_j=\kappa_e\\ e, j \in p,\, p \in \Gamma}}\{V^{j}(\tau(t, u^e))\}\Big\}\quad \forall \ e \in \mc E
\end{split}
\end{equation}
\begin{equation}\label{Valuefunc2}
V^0(t)=\min\limits_{\substack {e\in \mc E : \nu_e=o\\ e \in p,\, p \in \Gamma}}\{V^{e}(t)\}.
\end{equation}

In \eqref{Valuefunc1}, $\tau(t, u^e)$ is the time at which an agent entering link $e$ at time $t$ and choosing a control $u^e$ arrives in $\nu_j=\kappa_e$.\\
Note that the recursive definition of $V^e$ in \eqref{Valuefunc1} is not meaningless since the absence of oriented cycles in the network $\mc G$ prevents self-referring, that is it cannot occur that we define $V^e$ in terms of itself.
Also $V^0$ can be seen as associated to a fictitious link $e_0$ with null length, such that $\nu_{e_0}=\kappa_{e_0}=o$, through which the flow enters into the network.
Finally, observe that when the distributional evolution~$t\to \rho_e(t)$, i.e. the mass of the agents on the link $e$, is initially given then $V^e$ does not
depend on $\rho_e$.

In this paper, we assume that the physical traffic flow $f$ consist of indistinguishable
homogeneous agents which enter in the
network through the origin node, travel through it on the different paths and finally exit from the network through the
destination node. The relative appeal of the different paths to
the agents is modelled by a time-varying nonnegative vector $z(t)$ in the simplex 
\begin{equation}\label{def:Slambda}
\mc{S}_{\lambda(t)}=\left\{z\in\RR_{+}^{\vert\Gamma\vert}:\,1'z(t)=\lambda(t)\right\},
\end{equation}
where $\lambda(t) : [0, T] \to [0, \rho_{\max}]$ is the throughput, i.e., the total flow that goes through the network. We refer to the vector $z(t)$ as the current \textit{aggregate path preference} and let 	
\begin{equation}\label{flowzeta}
y^z(t)=Az(t)
\end{equation}
be the flow vector associated to it. 
The vector $z(t)$ is updated
as agents access global information about the current
congestion status of the whole network (that is embodied
by the mass vector $\rho(t)$) and it is also influenced by the 
vector $u(t)=\{u^e(t)\neq NaN: e \in \mc E\}$. Hereinafter, $\mc U$ is the set of these vectors.  Specifically, the cost perceived by each agent, traversing a link $e\in\mathcal E$, is given by \eqref{eqcost}.
The cost that an agent expects to incur along a path $p$ is
\begin{equation*}
J^p(t,u)=\sum_{e \in \mc E} A_{ep}J_e(t_e^p,u^e) 
\end{equation*}
where $t_e^p(t)$, when $e \in p$, is the time instant in which an agent, arriving in $t$ in the origin $o$ and following the path $p$, reaches $\nu_e$. In particular $t_e^p(t)=t$ when $\nu_e=o$ with $e \in p$. Clearly, $t_e^p$ depends on the  controls $u^j$ for each link $j \in \mc E$ that precedes the link $e$ along $p$. Arbitrarily, we also set $t_e^p(t)=t$ when $e \notin p$.
We denote with
\begin{equation*}
J(t,u)=\{J^p(t,u): p \in \Gamma\}
\end{equation*}
the vector of costs on all the paths $p \in \Gamma$. Then, we  assume that the path preferences are updated at some rate $\eta>0$, according to a noisy best response dynamics 
\begin{equation}\label{evolpi}
\dot{z}(t)=\eta(F^{(\beta)}(t, u)+Q^{(\beta)}(t, u)-z(t)),\ \ z(0)=z_0,
\end{equation}
which takes into accounts both the value of $\lambda(t)$ and $\dot{\lambda}(t)$.
Generally in the literature (see e.g., \cite{Como}, \cite{ComoMagg}) the throughput $\lambda$ is constant, hence in the formulation of the noisy best response dynamics $\dot{z}$ only the function $F^{(\beta)}$ appears. Here, as $\lambda$ varies over time, we need to introduce also the function $Q^{(\beta)}$ so that $z$ satisfies the constraint \eqref{def:Slambda}.
In particular, 
 for every fixed noise parameter $\beta>0$ the function $$F^{(\beta)}: [0, T]\times \mc U \to \RR^{\vert\Gamma\vert}_+$$ is the perturbed best response defined as follows:
\begin{equation}\label{bestresponse}
F^{(\beta)}(t, u)
=\frac{\lambda(t)\exp(-\beta(J(t,u)))}{1'\exp(-\beta(J(t,u)))},
\end{equation}
which provides an idealized description of the behavior of agents whose decisions are based on inexact information about the state of the network. Moreover it is continuous on $[0,T]\ \ \forall u \in \mc U$ fixed.
The function $Q^{(\beta)}:[0, T]\times \mc U \to \RR^{|\Gamma|}_+$ is defined as
\vspace{-0.3cm}
\begin{equation}\label{funzlambdapunto}Q^{(\beta)}(t, u)=\frac{\dot{\lambda}(t)\exp(-\beta(J(t,u)))}{1'\exp(-\beta(J(t,u)))}.
\end{equation}
We assume that \eqref{funzlambdapunto} has the same structure of \eqref{bestresponse} and it shows how the agents' preferences $z$ are updated taking into account not only the value of $\lambda(t)$ but also its time variation. For this reason we consider $\dot{\lambda}(t)$ and we assign to it the same weight of $\lambda(t)$ selecting the same fixed noise parameter $\beta$.
We now describe the \textit{local decision function} $G:\mc {S}_{\lambda(t)}\to\RR_+^{\vert\mc E\vert}$ characterizing
the fractions of agents choosing each outgoing link $e$ when traversing a non destination node $v$. It is given by
\begin{equation}\label{localchoice}
G_{e}(z)=
\begin{cases}
\displaystyle\frac{y_e^{z}}{\displaystyle\sum_{j \in \mathcal{E}: \nu_j=\nu_e}y_j^z} & \text{if}\quad{\displaystyle\sum_{j \in \mathcal{E}: \nu_j=\nu_e}y_j^z}>0\\
\,
\displaystyle\frac{1}{\vert\{j \in \mathcal{E}: \nu_j=\nu_e\}\vert}&  \text{if}\quad {\displaystyle\sum_{j \in \mathcal{E}: \nu_j=\nu_e}y_j^z}=0\,,
\end{cases}
\end{equation}
for each outgoing link $e$ in $\mc E$ and
 it is continuously differentiable on $ \mc {S}_{\lambda(t)}$. Equations  \eqref{localchoice} state that, at every node  $v \in \mc V$ the outflow is split proportionally to the flow vector $y^z$, if there is flow $y^z$ passing through node $v$, otherwise, the outflow is split uniformly among the immediately downstream links.  


Now, for every non-destination node $v$ and outgoing link $e \in \mc E$
conservation of mass implies that
\begin{equation}\label{sistemaccoppiato}
\dot{\rho}(t)=H(f(t), z(t))\,,\qquad \rho(0)=\rho_0
\end{equation}
where $H: \prod_{e\in \mc E}[0,C_e] \times \mc S_{\lambda}\to\RR^{|\mc E|}$ is defined as
\begin{equation}\label{H}
H_e(f, z):= G_{e}(z)\bigg(\lambda\delta_{\nu_e}^{(o)}+\sum_{j: \kappa_j=\nu_e}f_j\bigg)-f_e\,,\ \forall\ e\in\mc E,
\end{equation}
in which $\lambda(t)\delta_{\nu_e}^{(o)}$ accounts for the exogenous
inflow in the origin node $o$
and each component $f_e$ is 
\begin{equation}\label{flusso} f_e(t)=\rho_e(t) u^e(\rho)/ \ell_e,\end{equation} i.e., the flow on the link $e$ estimated by an agent entering in the very link at time $t$.\\
Next we introduce the hypotheses that we will use to prove the existence of a MF equilibrium:

 (H1)  the throughput $\lambda:[0,T]\to[0,+\infty[$ is $C^1([0,T])$ and bounded;

 (H2) 
	$\rho(0)=0$;
	
 (H3) $\varphi_e: [0, \rho_{\max}]\to [0, +\infty[$ are bounded, Lipschitz continuous and do not depend on the state variable $\theta_e$.\\ 
 
 
 \vspace{-0.3cm}
Note that $\rho(0)=0$ in (H2) means that no one is around the city at $t=0$, while (H3) implies that all agents in the same link at the same instant equally suffer the same congestion.
\begin{Remark}\label{ccp}
	Note that (H1) (H3) also imply that the control choice made by agents at the beginning of each link
	does not change as long as the agent remains in the same link, and it is constant in time. This is a feasible situation whenever the movement of an agent on a link cannot be interfered by the movements of agents that enter the same link after him. Note also that the same hypotheses imply that is  not convenient to go back along a link $e$, which in turn implies that the optimal control $u^e$ is always non-negative.
\end{Remark}

\subsection{Value Functions}
In order to define the MF equilibrium instead of coupling the two standard MF equations, i.e.,  Hamilton-Jacobi-Bellman equations and mass conservation ones \eqref{sistemaccoppiato}, with the preferences equations \eqref{evolpi} we write 
conditions equivalent 
in terms of the value functions. 
An agent standing at $\nu_{e}$, and hence at $\theta_e=0$, for $e \in \{e_4, e_5\}$ at time $t$ has two
possible choices: either staying at $\nu_{e}$ indefinitely or moving to reach $\kappa_{e}=d$ exactly at time $T$ (it is not optimal to reach $d$ before $T$ and
wait there for a positive time length, see, e.g., \cite{bagpes}). Accordingly, the controls~are
\begin{equation}\label{oc2}
u^{e}_1=0, \qquad u^{e}_2= \frac{\ell_{e}}{T-t}.
\end{equation} 
Hence, given the cost functional (\ref{eqcost}), we derive
\begin{equation}\label{eq:V1}
V^{e}(t)=\min\left\{ \alpha\ell_{e},\,\frac{1}{2}\frac{(\ell_{e})^{2}}{T-t}\right\} +\int_{t}^{T}\varphi_{e}\,ds 
\end{equation}
(note that we do not display the argument $\rho_{e}$ of $\varphi_{e}$ for $e \in \{e_4, e_5\}$. We use this convention also for the other $\varphi_e$ in the formulas below).
An agent standing at $\nu_{e_3}$ at time $t\in[0,T]$, has two possible
choices (meaning, the optimal behaviour may only be one of the following
two): staying in $\nu_{e_3}$ or moving to reach $\kappa_{e_3}$ at $\tau \in ]t, T]$. 
The  controls among which the agent chooses  and the corresponding value function are then, respectively
\begin{equation}\label{oc1}
u^{e_3}_1=0,\ \ u^{e_3}_2=\frac{\ell_{e_3}}{\tau-t},
\end{equation}
\vspace{-0.5cm}
\begin{align}\label{eq:V2}
V^{e_3}(t)= 
&\min\left\{ \alpha\left(\ell_{e_3}+\ell_{e_5}\right)+\int_{t}^{T }\varphi_{e_3}\,ds,\right.\\ &\left.\inf_{\tau\in ]t,T]}  \left\{\frac{1}{2}\frac{(\ell_{e_3})^{2}}{\tau-t}+\int_{t}^{\tau }\varphi_{e_3}\,ds+V^{e_5}(\tau)\right\}\right\}. \nonumber
\end{align}
An agent standing at $\nu_{e_1}$ at time $t$ may choose: staying in $\nu_{e_1}$ or to reach $\kappa_{e_1}$ at a certain $\tau \in ] t,T]$. 
Hence, the control is chosen among
\begin{equation}\label{oc4}
u^{e_1}_1=0,  \  u^{e_1}_2=  \frac{\ell_{e_1}}{\tau-t}.
\end{equation}
Consistently, one has
\begin{small}
\begin{align}\label{eq:V11}
&V^{e_1}(t)=\min\left\{ \alpha\min\left\{\ell_{e_1}+\ell_{e_4},\ell_{e_1}+\ell_{e_3}+\ell_{e_5}\right\}+\int_{t}^{T}\varphi_{e_1}\,ds,\right.\nonumber\\ 
& \left.\inf_{\tau\in ]t,T]}\left\{\frac{1}{2}\frac{(\ell_{e_1})^{2}}{\tau -t}+\int_{t}^{\tau }\varphi_{e_1}\,ds+ \min\{V^{e_3}( \tau),V^{e_4}(\tau)\}\right\}\right\}
\end{align}
\end{small}
Analogous arguments hold for computing $V^{e_2}(t)$ when an agent is standing at $\nu_{e_2}$, indeed in that case the controls are
\begin{equation}\label{oc3}
u_1^{e_2}=0, \quad u_2^{e_2}=\frac{\ell_{e_2}}{\tau-t}.
\end{equation}
Hence the value function is
\begin{align}\label{eq:V22}
V^{e_2}(t)&=\min\Bigg\{\alpha(\ell_{e_2}+\ell_{e_5})+\int_{t}^{T}\varphi_{e_2}\,ds,\\
& \inf_{\tau\in ]t,T]}\left\{\frac{1}{2}\frac{(\ell_{e_2})^{2}}{\tau -t}+\int_{t}^{\tau }\varphi_{e_2}\,ds+ V^{e_5}(\tau)\}\right\}\Bigg\}.\nonumber
\end{align}
Finally,
\begin{equation}\label{eq:V23}
V^0(t)=\min\{V^{e_1}(t), V^{e_2}(t)\}.
\end{equation}

Note that, the  optimal controls described in (\ref{oc1}), (\ref{oc4}), (\ref{oc3}) 
are detected  along with the  arrival time $\tau $ along the minimization process carried on in  (\ref{eq:V2}), (\ref{eq:V11}), (\ref{eq:V22}).
Also, when $\rho$ is given, the construction of the optimal controls is  performed  backwardly, starting from the problem (\ref{eq:V1}).
Finally, observe that $V^0(t)$ in \eqref{eq:V23} is determined without the necessity of computing any optimal control on the fictitious link $e_0$.
We summarize the previous discussion as follows.
\vspace{1mm}
\begin{Theorem} \label{th:valuefun} Suppose that $\rho$ is given and that (H1)--(H3) hold. Then, every value function $V^e : [0, \ell_e] \times [0,T] \to \mathbb{R}$, at $0$ i.e, at the beginning of each link, for all $e \in \cal E$ and for all $t\in[0,T]$,
	is determined through (\ref{eq:V1})-(\ref{eq:V22}).
	In addition, $V^e$
	is Lipschitz continuous with respect to time, with its Lipschitz constant independent of $\rho_e$.
\end{Theorem}
\vspace{1mm}
\begin{Remark}\label{rholip}
	Notice that, since the optimal controls are necessarily equi-bounded by a
	constant depending only on the parameters of the problem and since from (H1) follows that the inflow $\lambda(t)\delta_{\nu_e}^{(o)}$ is bounded and continuous,	
	by results on the mass conservation equations ((see, e.g., \cite{car}), follow that all $\rho_e$ are Lipschitz continuous, with Lipschitz constant $\tilde L$ depending only on the parameters of the problem. 
\end{Remark}

\section{EXISTENCE OF A MEAN FIELD EQUILIBRIUM}
\label{sec:3} {In this section we give a proof of the existence of a MF equilibrium.}
Let $L(w)$ be the Lipschitz constant of a function~$w$. As space to search for a fixed point, we choose
\begin{equation}\label{eq:X}
X=\left\{w:[0,T]\to[0,\rho_{\max}]:  {L(w)\le\tilde L},\, |w|\le \rho_{\max} \right\}^5,
\end{equation}
the Cartesian product five times of the space of  Lipschitzian functions with Lipschitz constant not greater than $\tilde L$  and overall bounded by $\rho_{\max}$, where $\tilde L$ is the constant introduced in Remark \ref{rholip}. 
Note that $X$ is convex and compact with respect to the uniform topology.

Fixed the noisy parameter $\beta>0$, we then search for a fixed point of the multi-function 
$\psi:~X\to~X$,  with
$\rho\mapsto \rho^\prime \in\psi(\rho)$ where 
$\rho^\prime$ is obtained as follows: (i)  $\rho$ is inserted in  (\ref{eq:V1})--(\ref{eq:V22}),  the optimal control $u$ is derived; (ii)  $u$ is inserted in (\ref{evolpi}) and the path preference vector $z$ is obtained; (iii) $\rho^\prime$ is derived from \eqref{sistemaccoppiato} after the computation of \eqref{localchoice} and \eqref{flusso}:
\vspace{-0.5cm}
	\begin{figure}[h!]
	\centering
	\begin{tikzpicture}
	[scale=1.3,auto=left,every node/.style={circle,draw=black!90,scale=.5,fill=white,minimum width=.5cm}]
	\node [scale=2, auto=center,fill=none,draw=none] (n0) at (0,0){$\rho$};
	\node [scale=2, auto=center,fill=none,draw=none] (n1) at (1,0){$u$};
	\node [scale=2, auto=center,fill=none,draw=none] (n2) at (2.2,0){$J(t,u)$};
	\node [scale=2, auto=center,fill=none,draw=none] (n3) at (3.4,0){$z$};
	\node [scale=2, auto=center,fill=none,draw=none] (n4) at (4.7,0){$H(f,z)$};
	\node [scale=2, auto=center,fill=none,draw=none] (n5) at (6,0){$\rho'$};
	\node [scale=2, auto=center,fill=none,draw=none] (n6) at (2.9,-0.6){$f$};	
	\foreach \from/\to in
	{n0/n1,n1/n6,n1/n2,n2/n3,n3/n4,n4/n5,n6/n4}
	\draw [-latex, right] (\from) to (\to); 
	\end{tikzpicture} 
\end{figure}
\vspace{-0.5cm}
\begin{Remark} Note that in general $\psi$ is not single valued, that is $\psi(\rho)$ is a nonsigleton subset of $X$. Indeed the optimal control may not  be unique as, for any fixed $t$, the minimization procedure in (\ref{eq:V1})--(\ref{eq:V22}) may return more than one minimizer. In particular, in (\ref{eq:V1})--(\ref{eq:V22}) this  may happen even along a whole time interval. 
Hence, when a multiplicity situation occurs, $\psi$ can be built in many ways,  as many as the different optimal behaviors. Indeed, for example, if the controls in \eqref{oc4} are both optimal we get two different flows $f_{e_1}$ and $f'_{e_1}$ (see \eqref{flusso}) and hence, through \eqref{sistemaccoppiato}, two different densities $\rho_{e_1}$ and $\rho'_{e_1}$ of which we should consider the convexification. Note that the controls' multiplicity does not influence the computation of $z$ and hence $G(z)$, but only the flows $f$ used in \eqref{sistemaccoppiato} to get the fixed point of $\psi$. To bound the times at which such multiplicities appear, we will obtain $\psi $ and its fixed point $\widetilde\rho$ through a limiting procedure.
\end{Remark}
\medskip
Fixed $\beta>0$, let $\left\{ \psi _{\varepsilon }\right\} _{\varepsilon >0}$ be a sequence  of functions
approximating $\psi $ in a suitable sense, and $\rho_\varepsilon$ its corresponding fixed points. The single function 
$\psi _{\varepsilon }$ is obtained through~(i)--(iii) above, with
the difference that in (ii), rather than choosing optimal controls, one
chooses $\varepsilon -$optimal controls and, along time, an $\varepsilon-$\emph{optimal stream}. Accordingly the path preference vector will be computed once the values of that $\varepsilon-$\emph{optimal stream} are given. We divide the construction into two steps. 

\medskip
\noindent \emph{Step 1: $\varepsilon$-optimal streams.}
\begin{Definition} 
	Assume for each link $e=(\nu_e, \kappa_e) \in \mc E$ to be in $\theta_e=0$, i.e., at node $\nu_e$, and let $u^{e}_{j}$, $j\in\{1,2\}$ be the controls defined through (\ref{oc2}) (\ref{oc1}) (\ref{oc4}) (\ref{oc3}). 
	Consider also a partition  $\tau^{e}=\{t^n\}_n$ of the interval $[0,T]$, {and fix $\varepsilon>0$}. 
	Then ${u}^{e }_{\varepsilon }$ is an
	$\varepsilon -$optimal stream for the link $e$ associated to the partition $\tau^{e}$ if
	$${u}^{e}_{\varepsilon }(s)=u^{e}_{j_n}(s), \ \ s\in[ t^{n},t^{n+1}[ $$  where  $u^{e}_{j_n}$ is  optimal at $t^{n}$ and $\varepsilon$-optimal at all $s\in] t^{n},t^{n+1}[$, {that is, it realizes the minimum cost up to an error not greater than $\varepsilon$.}
\end{Definition}

Notice that an $\varepsilon$-optimal stream associated to a general partition $\tau$ may or may not exist, but it certainly does when the partition is refined enough.  Indeed, considering  the functions involved in  the minimization process in (\ref{eq:V1}), (\ref{eq:V2}), (\ref{eq:V11}), (\ref{eq:V22}),
if the minimum is realized up to the error $\varepsilon$, the minima 
are attained within  intervals of type $[t+h_\varepsilon,T],$ for some suitable $h_\varepsilon>0$,
so that the functions cited above are Lipschitz continuous.
Denote by $L$ the greater of Lipschitz constants of these functions. Then a control $u^{e}_j$ optimal at $t$ remain $\varepsilon$-optimal at least along $[t,t+ \varepsilon/ L[$.  
In addition, for every link $e$, there may exist more than one $\varepsilon$-partition, as optimal controls determined in $\theta_e=0\,, \forall\, e \in \mc E,$  may be multiple. Anyway, the number of $\varepsilon$-partitions is overall bounded by a number $M_\varepsilon$, as the optimal controls at every $\theta_e=0$ are at most two. This argument proves the following.
\begin{Lemma}\label{taueps}
	Fixed $\varepsilon>0$, set $N=\max\{n\in{\mathbb N}: n<(TL)/\varepsilon\}+1$.  Consider the partition $\tau_\varepsilon$ of $[0,T]$ such that:
	\begin{itemize}
		\item[$(i)$] $t^0_\varepsilon=0$; $t^N_\varepsilon=T$;
		\item[$(ii)$] $t^n_\varepsilon=n L/ \varepsilon$, for all $n\in\{1,\dots, {N-1}\}$.
	\end{itemize}
	Then the set of $\varepsilon$-optimal streams associated to $\tau_\varepsilon$ is nonempty and finite.
\end{Lemma}

\vspace{1mm}


Now, for a fixed $\varepsilon>0$, consider the partition $\tau_\varepsilon$ defined in Lemma \ref{taueps}. Consider then the vector
$u_\varepsilon=\{u_\varepsilon^e: e \in \mc E\}$ whose components are  $\varepsilon$-optimal streams associated to  $\tau_\varepsilon$ in $\theta_e=0$.
As consequence of Lemma~\ref{taueps} the $\varepsilon$-optimal vectors $u_\varepsilon$ are a finite number of elements.
Note that once $u_\varepsilon $ is fixed,  all agents make the same decision, in time. Indeed, given the value of the $\varepsilon$-stream on each subinterval of $\tau_\varepsilon$, each agent, in that subinterval, can evaluates the total cost that he expects to incur in each path $p \in \Gamma$ and consequently it will be possible to compute 
$z$ through \eqref{evolpi}. Hereinafter, we refer to this $z$ as $z_\varepsilon$, since its value depends on $u_\varepsilon$.
The possibility for agents to split into fractions among different  vectors  $u_\varepsilon$  (that is, splitting among multiple optimal controls on  instants of $\tau_\varepsilon$), 
is given by an $\varepsilon-$split function $\mu_\varepsilon$ so defined.
\begin{Definition}\label{def:split}Consider the partition $\tau_\varepsilon$. An \emph{$\varepsilon$-split function} is a vector $\mu_\varepsilon=(\mu_j^{(e)}: \,j=1,2, \, e \in \mc E)\in L^\infty(0,T)^{2\times \vert\cal E\vert}$ whose components $\mu^{(e)}_j$,
	  called \emph{split fractions}, are constant on subintervals induced by $\tau_\varepsilon$ and satisfy:
	\begin{itemize}
		\item[$(i)$]  $\mu^{(e)}_j(s)\ge 0,$ for all $s\in[0,T]$,
		and $\mu^{(e)}_j(t)=0$ if $u^{e}_j$ is not optimal at $(0,t)$ for $t\in\tau_\varepsilon$;
		\item[$(ii)$] $\sum_{j=1}^2\mu^{(e)}_j(s)=1$ for all $s\in[0,T]$.
	\end{itemize}
\end{Definition}
We can see $\mu_j^{(e)}$ as the fraction of agents who choose the control $u_j^e$ in the link $e \in \mc E$.

\vspace{1mm}
\emph{Step 2: Construction of $\psi_\varepsilon(\rho)$.}

Let $\beta>0$, $\varepsilon>0$ and $\rho=(\rho_{e_1},\rho_{e_2},\rho_{e_3},\rho_{e_4},\rho_{e_5})\in X$ be fixed. Let also $\tau_\varepsilon$ be the partition of $[0,T]$ described above. We now built the multifunction
{$\psi_\varepsilon(\rho)\subseteq X$} with compact and convex images and closed
graph, to which later we can apply Kakutani fixed point theorem. 

(a) We define $\tilde\psi_\varepsilon(\rho)\subseteq X$ as the finite set of vectors $\rho^{\prime }=~(\rho_{e_1}^{\prime},\rho_{e_2}^{\prime}, \rho_{e_3}^{\prime},\rho_{e_4}^{\prime},\rho_{e_5}^{\prime}) \in X$ constructed in the following
way. We consider an $\varepsilon-$optimal vector
$u_{\varepsilon}$, associated to $\tau_\varepsilon$ in $\theta_e=0$ for every $e \in \mc E$.
Then, we use $u_{\varepsilon}$ to determine the path preference vector $z_\varepsilon$.
Finally,  given the arrival flow of agents $\lambda(t)\delta^{(0)}$ and computed the flow $f$ associated to the initially given $\rho$ and to $u_\varepsilon$ \eqref{flusso}, we solve \eqref{sistemaccoppiato} 
whose solution is the total mass~$\rho^{\prime }$. We repeat the construction  for all possible (and finite, by Lemma \ref{taueps}) choices  $u_\varepsilon$ and call the set of all outcomes $\tilde\psi_\varepsilon(\rho)$ which is a finite set.

(b)	We define $\psi_\varepsilon(\rho)\subseteq X$ as follows. 
We consider an arbitrary $\varepsilon$-split function $\mu_\varepsilon$ and assume that at every point of $\tau_\varepsilon$ the outflow of agents $f_e$ from every link $e$ gets split in fractions $\mu_1^{(e)}f_e$, $\mu_2^{(e)}f_e$ among the links $e'$ immediately downstream the link $e$.
At the end of the process the output is $\rho^\prime$.
We define $\psi_\varepsilon(\rho)$ as the set of all $\rho^\prime$  generated by all possible choices of the $\varepsilon$-split function $\mu_\varepsilon$. Clearly $\psi_\varepsilon(\rho)\supset \tilde \psi_\varepsilon(\rho)$.

%
\begin{Lemma}\label{lem:2} The set $\psi_\varepsilon(\rho)$ is a {non-empty} convex and compact subset of $X$, {for any $\rho \in X$}.  
	Moreover, the map $\rho\mapsto \psi_\varepsilon(\rho)$ has closed graph and
	it has a fixed point $\rho_\varepsilon\in X$.
\end{Lemma} 
\begin{proof} We preliminarily observe that the set $\psi_\varepsilon(\rho)$ is non-empty by construction. In addition, it is also closed and, hence, it is compact since $X$ is compact. Moreover, it is possible to prove that $\psi_\varepsilon(\rho)$ is the convex hull of  $\tilde\psi_\varepsilon(\rho)$ which is the set of all $\rho$ whose corresponding $\mu_\varepsilon$ are obtained through all possible $\varepsilon-$streams.
	Observe that the set $\tilde\psi_\varepsilon(\rho)$ is finite and includes at most $M_\varepsilon$ elements. Hence the set $\psi_\varepsilon(\rho)$ has a finite number of extremal points. This fact, together with the regularity assumptions (H1) and (H3), implies, reasoning as in \cite{bafama} but taking into account $z$, that the multifunction $\rho\mapsto\psi_\varepsilon(\rho)$ has closed graph. Hence, by the Kakutani theorem, there exists a fixed point $\rho_\varepsilon\in\psi_\varepsilon(\rho_\varepsilon)$.		
	%
	%
	%
\end{proof}
Hereinafter, we denote by $\rho_\varepsilon$ a fixed point for $\psi_{\varepsilon}(\rho)$, i.e., a total mass that satisfies $\rho_\varepsilon \in \psi_{\varepsilon}(\rho_\varepsilon)$.

Before stating the existence of a MF equilibrium, 
we introduce the following definition that help restrict the equilibrium concept to the purpose of our problem.
\begin{Definition}\label{def:MFGE}
	Let $\psi$ and $\psi_\varepsilon$ be the functions described at the beginning of Section \ref{sec:3} with $\beta>0$ fixed.
	\begin{itemize}	
		\item	A \textit{$\varepsilon$-MF equilibrium} is a total mass~$\rho_\varepsilon \in X$ that satisfies $\rho_\varepsilon \in \psi_{\varepsilon}(\rho_\varepsilon)$. 
		\item	A \textit{MF equilibrium} is a total mass $\rho\in X$ that satisfies $\rho \in \psi(\rho)$.
	\end{itemize}	
\end{Definition}
Note that  $\rho \in \psi(\rho)$ implies 
that $\rho$ induces a set of optimal controls as in~(\ref{oc2}),~(\ref{oc1}),~(\ref{oc4}),~(\ref{oc3}) used both to compute the corresponding path preference vector $z$ and to define the fractions of flows according to a split function $\mu$. Such a function is given as in Definition \ref{def:split} but with the difference that $\mu$ is not linked to any $\varepsilon$-partition, and its components are not necessarily piecewise
constant. Finally, using \eqref{sistemaccoppiato} we get again $\rho$.
\begin{Theorem}\label{th:MFE}
	Assume  (H1)--(H3). Then there exists a MF equilibrium.
\end{Theorem}
\begin{proof} We proceed similarly to Theorem 2 of \cite{bafama}, but with the suitable changes due to the presence of the agent's preference dynamics \eqref{evolpi}.
\end{proof}
\begin{Remark}
In the construction of the multi-function $\psi$ we have fixed a priori the noisy parameter $\beta>0$. Hence, we actually found a multi-function $\psi_\beta$ such that $\rho_\beta \in \psi_\beta(\rho_\beta)$. Given  the simple structure of \eqref{evolpi} and \eqref{sistemaccoppiato} and to the properties of \eqref{bestresponse} and \eqref{funzlambdapunto}, it seem reasonable that for $\beta\to +\infty$ there exist a multi-function $\bar\psi$ and the corresponding fixed point $\bar\rho$ such that
$\lim_{\beta\to +\infty}\psi_\beta=\bar\psi$ and $ \lim_{\beta\to +\infty}\rho_\beta=\bar\rho$.
%
The answer seem positive, however, we have not yet covered all the details which we will pursue in future research.
\end{Remark}
%
\section{BI-LEVEL OPTIMIZATION}\label{sec:4}
In this section, we discuss how the results introduced in the previous sections can be used by a central authority (CA), e.g., the city hall, interested in controlling the value of the mass $\rho$.
For any given input flow $\lambda(t) \delta^{(0)}, t \in [0,T]$, heritage CA may be consider sustainable for their historical centers that mass of agents is closed to some reference value $\widetilde \rho$.  Typically, a CA may slow down (and sometimes also speed up) the agent flows by intervening on the width of the streets with mobile barriers. 
Formally, we can assume that the CA can consider congestion cost functions of the form
\begin{align*}
\varphi_e(\rho_e)=\alpha'_e\rho_e(t)+\alpha''_e \quad \forall\ e \in \mc E, 
\end{align*}
where $(\alpha', \alpha'')$ are parameters in a compact $K \in \RR^5\times \RR^5 $.

Then, the CA is interested in solving the bi-level problem:
\begin{align}
&\displaystyle\inf_{\left(\alpha', \alpha''\right)}\left(\sup_{\rho}\Vert\rho-\widetilde{\rho}\Vert_{\infty}\right)\nonumber \\
& \text{subject to} \quad \inf_{u(\rho) \in U_{\alpha', \alpha''}}\hat{J}(t,u) \label{eq:bil}
\end{align}
with $\hat{J}(t,u)=\{J_e(t, u^e) : e \in \mc E\}$ and
where the CA chooses the value of pair $(\alpha', \alpha'')$ in the light of the expected response of the agents. 
We claim that, if the set $ U_{\alpha', \alpha''}$ is a compact set of controls depending by $(\alpha', \alpha'')$, then
the existence of a pair $(\alpha', \alpha'')^*$ solution of~(\ref{eq:bil}) can be proved provided that a unique MFG equilibrium exists for each fixed $(\alpha', \alpha'') \in K$. 

\section{CONCLUSIONS}\label{sec:5}
In this paper the existence of a mean-field equilibrium for a
traffic flow model is provided and a bi-level optimization problem is formulated. In future research we plan to refine the optimization problem from both a theoretical and an applicative perspective, and to consider different objectives of central authority. Moreover a comparison between our MF model and the Wardrop one also through numerical simulations will be provided.
A further step could be to consider a more general networks that includes oriented cycles. For this kind of networks the backward approach used for the resolution of equations \eqref{eq:V1}-\eqref{eq:V23} cannot immediately applied.

\addtolength{\textheight}{-12cm}   

\end{document}